\documentclass[11pt,reqno]{amsart}
\usepackage{lmodern}
\usepackage{microtype}
\usepackage[a4paper,margin=29mm]{geometry}
\usepackage{amssymb}
\usepackage{xcolor}
\usepackage[colorlinks=true,linkcolor=blue!45!black,citecolor=blue!45!black,urlcolor=blue!45!black]{hyperref}
\hypersetup{pdftitle={Quantitative Merino--Welsh inequalities for joins},pdfauthor={Jungang Chen and Jiaxin Xie},pdfsubject={Tutte polynomial, graph joins, effective resistance, spanning trees}}
\newtheorem{theorem}{Theorem}[section]
\newtheorem{lemma}[theorem]{Lemma}
\newtheorem{proposition}[theorem]{Proposition}
\newtheorem{corollary}[theorem]{Corollary}
\numberwithin{equation}{section}
\newcommand{\one}{\mathbf{1}}

\title{Quantitative Merino--Welsh inequalities for joins}
\author{Jungang Chen}
\author{Jiaxin Xie}
\thanks{\textsuperscript{*}Corresponding author}
\makeatletter
\let\addresses\@empty
\renewcommand{\@setauthors}{%
  \par\vspace{1.2ex}%
  \noindent\makebox[\textwidth][c]{%
    \begin{minipage}[t]{0.46\textwidth}
      \centering
      {\small\scshape Jungang Chen\textsuperscript{*}\par}
      \vspace{0.5ex}
      {\footnotesize School of Mathematical Sciences\par
       Xiamen University\par
       Xiamen, P.~R.~China\par
       \href{mailto:jgchen@stu.xmu.edu.cn}{jgchen@stu.xmu.edu.cn}\par}
    \end{minipage}%
    \hspace{0.04\textwidth}%
    \begin{minipage}[t]{0.46\textwidth}
      \centering
      {\small\scshape Jiaxin Xie\par}
      \vspace{0.5ex}
      {\footnotesize School of Mathematical Sciences\par
       Xiamen University\par
       Xiamen, P.~R.~China\par
       \href{mailto:jiaxinxie@stu.xmu.edu.cn}{jiaxinxie@stu.xmu.edu.cn}\par}
    \end{minipage}%
  }%
  \par\vspace{1ex}%
}
\makeatother
\subjclass{Primary 05C31; Secondary 05C05, 05C20, 05C50}
\keywords{Tutte polynomial, Merino--Welsh inequality, graph joins, effective resistance, spanning trees, acyclic orientations}

\begin{document}
\begin{abstract}
For a connected graph \(G\), let
\[
Q(G)=\frac{T(G;2,0)T(G;0,2)}{T(G;1,1)^2}.
\]
We obtain quantitative lower bounds for \(Q\) under the graph join operation. If \(A\) and \(B\) are arbitrary simple graphs of orders \(3\le a\le b\), then \(Q(A\vee B)\) admits an explicit lower bound depending only on \(a\) and \(b\), and this bound is strictly greater than \(1\). We further quantify the improvement produced by edges inside the two factors. For every simple graph \(F\), with \(n=|V(F)|+2\ge4\), we prove
\[
Q(K_2\vee F)\ge \frac{27}{n^2}\left(\frac32\right)^{n-4}.
\]
Consequently, every join of at least three nonempty factors, and every complete multipartite graph with at least one edge and no cut edges, satisfies the strict multiplicative Merino--Welsh inequality. The proofs combine orientation estimates with spanning-tree comparisons based on effective resistance and block elimination.
\end{abstract}
\maketitle

\section{Introduction and main results}

Unless otherwise stated, all graphs in this paper are finite and simple, and our graph-theoretic terminology and notation follow Bondy and Murty~\cite{BondyMurty2008}.
In particular, $v(G)=|V(G)|$ and $e(G)=|E(G)|$ denote the order and the number of edges of a graph $G$, respectively, while $\alpha(G)$ and $\alpha^{*}(G)$ denote the numbers of acyclic orientations and \emph{totally cyclic orientations} of $G$, respectively. For a connected graph $G$, $\tau(G)$ denotes its number of spanning trees. These counts are evaluations of the Tutte polynomial:
\[
\alpha(G)=T(G;2,0),\qquad
\alpha^{*}(G)=T(G;0,2),\qquad
\tau(G)=T(G;1,1).
\]

Concerning these counts, Merino and Welsh~\cite{MerinoWelsh1999} proposed the following conjecture in 1999: for every connected graph $G$ without cut edges,
\[
\max\{\alpha(G),\alpha^{*}(G)\}\ge \tau(G).
\]
Conde and Merino~\cite[Conjecture~2.2]{CondeMerino2009} proposed a stronger multiplicative form in 2009, namely, under the same assumptions,
\[
\alpha(G)\alpha^{*}(G)\ge\tau(G)^2.
\]
Let $Q(G)=(\alpha(G)\alpha^{*}(G))/\tau(G)^2$; then the preceding inequality is equivalent to $Q(G)\ge1$. This multiplicative inequality clearly implies the maximum inequality above.

Thomassen~\cite{Thomassen2010} proved the maximum form under suitable sparsity or density conditions. The multiplicative conjecture was subsequently verified for several important graph classes, including 2-connected threshold graphs~\cite{CondeMerino2009},
series-parallel graphs without cut edges~\cite{NobleRoyle2014}, and graphs of pathwidth at most three without cut edges~\cite{Ok2016}.
In the matroid setting, the corresponding multiplicative inequality has also been proved for loopless and coloopless lattice path matroids~\cite{KnauerEtAl2018}
and split matroids~\cite{FerroniSchroter2023}. These studies largely concern the lower bound $1$ for graphs or matroids with special structures, whereas here we focus on lower bounds for $Q(G)$ when two arbitrary graphs are combined by the join operation and obtain bounds stronger than $1$. These bounds depend on both the orders of the factors and their numbers of internal edges.

Recall that the \emph{join} of graphs $A$ and $B$, denoted by $A\vee B$, is obtained from the disjoint union of copies of the two graphs by adding all edges between their vertex sets; the graphs $A$ and $B$ are called the two factors of $A\vee B$. Throughout, we assume that each factor has at least one vertex. An empty graph means a graph with no edges.

For an integer $s\ge2$, write $\lambda_s=(s+1)(2^s-2)/s^2$.
Let $D_3=34/15$; for $a\ge4$, write
\[
D_a=\frac{34}{15}\prod_{j=3}^{a-1}
\left(\lambda_j\lambda_{j+1}
\left(\frac{j}{j+1}\right)^{4j-2}\right).
\]
We first consider joins whose factors both have order at least three. Using electrical networks to bound the growth in the number of spanning trees, we obtain a lower bound that depends only on the factor orders, not on their internal edge sets:

\begin{theorem}\label{thm:pair}
Let $A,B$ be simple graphs whose orders satisfy $3\le a\le b$. Then
\begin{equation}\label{eq:pair-bound}
Q(A\vee B)\ge
D_a\lambda_a^{b-a}\left(\frac ab\right)^{2a-2}.
\end{equation}
\end{theorem}

Write $\eta_s=2s^2/(s+2)^2$. Next, we quantify the effect of internal edges of the factors on the lower bound for $Q$:

\begin{theorem}\label{thm:internal}
For simple graphs $A,B$ of orders $a,b\ge2$, respectively, we have
\begin{equation}\label{eq:internal-bound}
Q(A\vee B)\ge
Q(K_{a,b})\,\eta_b^{e(A)}\eta_a^{e(B)}.
\end{equation}
When $a,b\ge5$, we further have
\[
Q(A\vee B)\ge Q(K_{a,b})\left(\frac{50}{49}\right)^{e(A)+e(B)}.
\]
In this case, $Q(A\vee B)\ge Q(K_{a,b})$, with equality if and only if both factors are empty graphs.
\end{theorem}

We next consider joins of the form $K_2\vee F$:

\begin{theorem}\label{thm:universal}
Let $F$ be a simple graph with at least one vertex, and let $G=K_2\vee F$ and $n=v(G)$.
When $n\ge4$,
\begin{equation}\label{eq:universal-bound}
Q(G)\ge \frac{27}{n^2}\left(\frac32\right)^{n-4}.
\end{equation}
When $n=3$, we have $G=K_3$ and $Q(G)=4/3$.
\end{theorem}

Section~\ref{sec:comparison} collects the necessary preliminaries and comparison tools.
Section~\ref{sec:proofs} proves the three theorems above in turn and applies them to joins of at least three factors and to complete multipartite graphs with at least one edge and no cut edges.

\section{Growth in the number of spanning trees and comparisons of \texorpdfstring{$Q$}{Q}}\label{sec:comparison}
\subsection{Adding a vertex}
Throughout this subsection, we assume that $G$ is obtained from a connected graph $H$ by adding a new vertex $v$
with neighbour set $S\subseteq V(H)$, where $d=|S|\ge2$.

For acyclic orientations, the work of Kahal\'e and Schulman~\cite[p.~12, Lemma~4, in the author manuscript]{KahaleSchulman1996}
gives the following special case with the parameter $w=1$:

\begin{lemma}\label{lem:acyclic-orientations}
We have $\alpha(G)\ge(d+1)\alpha(H)$.
\end{lemma}

For totally cyclic orientations, the work of Conde and Merino~\cite[Lemma~3.2, p.~83]{CondeMerino2009} readily yields the following:

\begin{lemma}\label{lem:cyclic-orientations}
We have $\alpha^{*}(G)\ge(2^d-2)\alpha^{*}(H)$.
\end{lemma}

Next, we use electrical networks to give an upper bound on the growth in the number of spanning trees. Assign unit resistance, or equivalently unit conductance, to each edge.
Denote the Laplacian matrix of a graph $H$ by $\mathbf C(H)=\mathbf D(H)-\mathbf A(H)$,
where $\mathbf D(H)$ is the diagonal degree matrix and $\mathbf A(H)$ is the adjacency matrix.
For a weighted graph, $\mathbf A(H)$ records the edge conductances, and each diagonal entry of $\mathbf D(H)$ is the sum of the conductances of the edges incident with the corresponding vertex.
Below, $\mathbf I_k$ and $\one_k$ denote the identity matrix of order $k$ and the all-ones column vector of dimension $k$, respectively; subscripts are omitted when the dimensions are clear.

When a unit current is injected at vertex $x$ and withdrawn at $y$, the potential difference between the two vertices is called the \emph{effective resistance}, denoted by $r_{xy}(H)$. Since $H$ is connected, by~\cite[equation~(3)]{DorflerBullo2013},
it has the matrix expression
\[
r_{xy}(H)=(\mathbf e_x-\mathbf e_y)^T \mathbf C(H)^+(\mathbf e_x-\mathbf e_y),
\]
where $\mathbf e_x,\mathbf e_y$ are the corresponding standard unit column vectors,
and $\mathbf C(H)^+$ is the Moore--Penrose pseudoinverse.

Let $U=V(H)\setminus S$. Order the vertices according to $S,U$ and write $\mathbf C(H)$ in block form as
\[
\mathbf C(H)=
\begin{pmatrix}
\mathbf C_{SS}&\mathbf C_{SU}\\
\mathbf C_{US}&\mathbf C_{UU}
\end{pmatrix}.
\]
When $U$ is nonempty, $\mathbf C_{UU}$ is positive definite.
Thus, block elimination with respect to $\mathbf C_{UU}$~\cite{Zhang2005} gives the Kron-reduced matrix onto $S$:
\begin{equation}\label{eq:kron}
\mathbf K=\mathbf C_{SS}-\mathbf C_{SU}\mathbf C_{UU}^{-1}\mathbf C_{US}.
\end{equation}
If $U$ is empty, set $\mathbf K=\mathbf C(H)$. Denote all the positive eigenvalues of $\mathbf K$ by $\mu_2,\ldots,\mu_d$. By~\cite[Lemma~II.1 and Theorem~III.8(1)]{DorflerBullo2013}, the matrix $\mathbf K$ is the Laplacian matrix of a connected weighted graph on $S$, and the effective resistance between any two vertices in $S$ is unchanged by the reduction;
that is, for any $x,y\in S$, we have $r_{xy}(H)=r_{xy}(\mathbf K)$. Write
\[\overline r_H(S):=\frac{2}{d(d-1)}
\sum_{\{x,y\}\subseteq S}r_{xy}(H)\]
for the average effective resistance over pairs of vertices in $S$. Combining the resistance invariance under Kron reduction with the identities of Ghosh, Boyd, and Saberi~\cite[Section~2.5, equations~(13) and~(15)]{GhoshBoydSaberi2008} gives the following:

\begin{lemma}\label{lem:average-resistance}
The average effective resistance between pairs of vertices in $S$ satisfies
$\overline r_H(S)=\frac{2}{d-1}\sum_{i=2}^d\frac1{\mu_i}$.
\end{lemma}

\begin{lemma}
The ratio of the numbers of spanning trees satisfies
\begin{equation}\label{eq:tree-growth}
R:=\frac{\tau(G)}{\tau(H)}
=d\prod_{i=2}^d\left(1+\frac1{\mu_i}\right).
\end{equation}
\end{lemma}
\begin{proof}
If $U$ is nonempty, delete from $\mathbf C(G)$ the row and column corresponding to the new vertex $v$.
By the Matrix--Tree Theorem~\cite[Corollary~20.16]{BondyMurty2008}, we obtain
\[
\tau(G)=\det
\begin{pmatrix}
\mathbf C_{SS}+\mathbf I_d&\mathbf C_{SU}\\
\mathbf C_{US}&\mathbf C_{UU}
\end{pmatrix}
=\det\mathbf C_{UU}\det(\mathbf K+\mathbf I_d)
=\det\mathbf C_{UU}\prod_{i=2}^d(1+\mu_i).
\]
On the other hand, choose any $s\in S$, delete from $\mathbf C(H)$ the row and column corresponding to $s$,
and partition the resulting principal submatrix into blocks according to $S\setminus\{s\},U$.
Let $\mathbf K^{(s)}$ be the principal submatrix obtained from $\mathbf K$ by deleting the row and column corresponding to $s$. Again by the Matrix--Tree Theorem,
\[
\tau(H)=\det\mathbf C_{UU}\det\mathbf K^{(s)}.
\]
The matrix $\mathbf K$ is the Laplacian matrix of a connected weighted graph on $S$.
By the spectral form of the weighted Matrix--Tree Theorem~\cite[Section~1.1, Theorem~1.2]{DuvalKlivansMartin2009},
we have $\det\mathbf K^{(s)}=\frac{1}{d}\prod_{i=2}^d\mu_i$, and hence
\[
\tau(H)=\frac{\det\mathbf C_{UU}}d\prod_{i=2}^d\mu_i.
\]
Dividing the two expressions proves the claim. If $U$ is empty, take $\det\mathbf C_{UU}$ to be one in all the expressions above.
\end{proof}

\begin{lemma}\label{lem:tree-growth-monotonicity}
Let $H_0$ be a connected spanning subgraph of $H$, and let $G_0$ be obtained from $H_0$ by adding a new vertex $v$
with the same neighbour set $S$. Then
$\frac{\tau(G)}{\tau(H)}\le\frac{\tau(G_0)}{\tau(H_0)}$.
\end{lemma}
\begin{proof}
Let $\mathbf K_0$ and $\mathbf K$ be the Kron-reduced matrices of $H_0$ and $H$ onto $S$, respectively.
When $U$ is nonempty, take $\xi\in\mathbb R^d$ and $\eta\in\mathbb R^{|U|}$
as coordinate vectors on $S$ and $U$, respectively, and write $\zeta=(\xi^T,\eta^T)^T$. Expanding the block matrix expression gives
\[
\begin{pmatrix}\xi\\\eta\end{pmatrix}^{T}
\mathbf C(H)
\begin{pmatrix}\xi\\\eta\end{pmatrix}
=\xi^T\mathbf C_{SS}\xi+2\eta^T\mathbf C_{US}\xi+\eta^T\mathbf C_{UU}\eta.
\]
Fix $\xi$. Since $\mathbf C_{UU}$ is positive definite, differentiation with respect to $\eta$ shows that the right-hand side
attains its unique minimum at $\eta_*=-\mathbf C_{UU}^{-1}\mathbf C_{US}\xi$.
Substituting $\eta_*$ and using~\eqref{eq:kron}, we obtain
\[
\xi^T\mathbf K\xi=\min_{\eta\in\mathbb R^{|U|}}
\begin{pmatrix}\xi\\\eta\end{pmatrix}^{T}
\mathbf C(H)
\begin{pmatrix}\xi\\\eta\end{pmatrix}.
\]
The same argument applies to $H_0$. On the other hand, for any vertex coordinate vector $\zeta$ as above,
\[
\zeta^T\mathbf C(H)\zeta=\sum_{\{a,b\}\in E(H)}(\zeta_a-\zeta_b)^2.
\]
Thus, for any $\xi,\eta$, we have $\zeta^T\mathbf C(H_0)\zeta\le\zeta^T\mathbf C(H)\zeta$.
Fixing $\xi$ and minimizing both sides over $\eta$ gives $\xi^T\mathbf K_0\xi\le\xi^T\mathbf K\xi$.
When $U$ is empty, this inequality follows directly from $\mathbf K_0=\mathbf C(H_0)$ and $\mathbf K=\mathbf C(H)$.
By the min--max principle for eigenvalues~\cite[Chapter~4]{HornJohnson2012}, the eigenvalues of $\mathbf K_0$ on $\one_S^\perp$, listed in nondecreasing order,
are no greater than the corresponding eigenvalues of $\mathbf K$. Combining this with~\eqref{eq:tree-growth} proves the result.
\end{proof}

\begin{proposition}\label{prop:resistance}
With the notation above, we have
\[
R\le d\left(1+\frac{\overline r_H(S)}2\right)^{d-1},\qquad
Q(G)\ge Q(H)\frac{d^2\lambda_d}{R^2}.
\]
Equality holds in the first inequality if and only if all positive eigenvalues of $\mathbf K$ are equal.
\end{proposition}
\begin{proof}
Apply the arithmetic--geometric mean inequality to~\eqref{eq:tree-growth}, and then use Lemma~\ref{lem:average-resistance} to obtain
\[
R\le d\left(\frac{\sum_{i=2}^d(1+\frac1{\mu_i})}{d-1}\right)^{d-1}
=d\left(1+\frac{\overline r_H(S)}2\right)^{d-1}.
\]
Equality holds if and only if all $\mu_i$ are equal.
Multiplying the inequalities in Lemmas~\ref{lem:acyclic-orientations} and~\ref{lem:cyclic-orientations}
and then dividing by $\tau(G)^2$ gives the second inequality.
\end{proof}

\subsection{Adding an edge}
The graph obtained from $H$ by adding a new edge with ends $u$ and $v$ is denoted by $H+uv$.
\begin{lemma}\label{lem:edge-growth}
Let $H$ be connected and have no cut edges, and let $u,v\in V(H)$ be nonadjacent. Then
\[
\frac{\tau(H+uv)}{\tau(H)}=1+r_{uv}(H),\qquad
Q(H+uv)\ge Q(H)\frac{2}{(1+r_{uv}(H))^2}.
\]
\end{lemma}
\begin{proof}
Since $u,v$ are nonadjacent and $H$ is connected, we may choose $z\in V(H)\setminus\{u,v\}$.
Let $\mathbf M$ be the principal submatrix obtained from $\mathbf C(H)$ by deleting the row and column corresponding to $z$,
and let $b$ be the vector obtained from $\mathbf e_u-\mathbf e_v$ by deleting the coordinate corresponding to $z$.
Since $H$ is connected, $\mathbf M$ is invertible; when the edge $uv$ is added, the corresponding principal submatrix is $\mathbf M+bb^T$.
By the Matrix--Tree Theorem,
\[
\frac{\tau(H+uv)}{\tau(H)}
=\frac{\det(\mathbf M+bb^T)}{\det\mathbf M}
=1+b^T\mathbf M^{-1}b.
\]
Set the potential at $z$ to zero. When a unit current is injected at $u$ and withdrawn at $v$, the potential vector $\varphi$ on the remaining vertices satisfies
$\mathbf M\varphi=b$. Hence
\[
b^T\mathbf M^{-1}b=b^T\varphi=\varphi_u-\varphi_v=r_{uv}(H),
\]
which gives the first identity.

For any acyclic orientation of $H$, if both directions of the new edge create a directed cycle, then the original orientation contains
a directed path from $u$ to $v$ and one from $v$ to $u$, a contradiction. Thus, every acyclic orientation has at least one acyclic extension, so
$\alpha(H+uv)\ge\alpha(H)$.
On the other hand, in any totally cyclic orientation of $H$, the two ends of every edge are mutually reachable; since $H$ is connected,
any two vertices are mutually reachable. Thus, the new edge belongs to a directed cycle in either direction, and every original edge still belongs to a directed cycle.
Hence, every totally cyclic orientation has two totally cyclic extensions, so $\alpha^{*}(H+uv)\ge2\alpha^{*}(H)$.
Combining this with the first identity and the definition of $Q$ gives the second inequality.
\end{proof}

\section{Proofs of the main theorems}\label{sec:proofs}

\subsection{Proof of Theorem~\ref{thm:pair}}
By retaining only the edges between the factors, we obtain an estimate for the growth in the number of spanning trees that is independent of the internal structure.

\begin{lemma}\label{lem:bipartite}
Let $H=X\vee Y$, where $v(X)=t\ge1$ and $v(Y)=q\ge1$.
Take a new vertex $v$ adjacent to every vertex of $Y$ and to exactly $r$ vertices of $X$,
where $0\le r\le t$ and $d=q+r\ge2$. Denote the resulting graph by $H+v$. Then
\begin{equation}\label{eq:bipartite}
\frac{\tau(H+v)}{\tau(H)}\le
\begin{cases}
q(1+1/t)^{q-1},&r=0,\\
(d+1)(1+1/t)^{q-1}(1+1/q)^{r-1},&r\ge1.
\end{cases}
\end{equation}
\end{lemma}
\begin{proof}
By Lemma~\ref{lem:tree-growth-monotonicity}, the ratio $\tau(H+v)/\tau(H)$ is at most the corresponding ratio for $K_{t,q}$ with the same neighbour set of $v$.
Let $\mathbf J_{a,b}$ be the all-ones $a\times b$ matrix, and set $\mathbf J_a=\mathbf J_{a,a}$.
When $r=0$, the Kron-reduced matrix onto $Y$ is
$t(\mathbf I_q-\mathbf J_q/q)$, whose positive eigenvalue is $t$ with multiplicity $q-1$.

When $r\ge1$, the Kron-reduced matrix onto $Y$ together with the selected $r$ vertices of $X$ is
\[
\begin{pmatrix}
t\mathbf I_q-\dfrac{t-r}{q}\mathbf J_q&-\mathbf J_{q,r}\\
-\mathbf J_{r,q}&q\mathbf I_r
\end{pmatrix}.
\]
Its positive eigenvalues are $t,q$ and $d$, with multiplicities $q-1,r-1$ and $1$, respectively. Substituting these values into~\eqref{eq:tree-growth} gives the result.
\end{proof}

For integers $q\ge3$ and $t\ge1$, write
$g_q(t):=\lambda_q(t/(t+1))^{2q-2}$.

\begin{lemma}\label{lem:step}
Under the assumptions of Lemma~\ref{lem:bipartite}, if $q\ge3$, then
$Q(H+v)\ge Q(H)g_q(t)$.
\end{lemma}
\begin{proof}
If $r=0$, the result follows by combining~\eqref{eq:bipartite} with Proposition~\ref{prop:resistance}.
If $r\ge1$, combining these two inequalities gives $Q(H+v)\ge Q(H)g_q(t)F(q,r)$, where
\[
F(q,r):=\frac{q^2(2^{q+r}-2)}{(q+r+1)(q+1)(2^q-2)}
\left(\frac q{q+1}\right)^{2r-2}.
\]
Thus, it suffices to prove that $F(q,r)\ge1$. When $q\ge4$,
\[
F(q,1)>\frac{2q^2}{(q+1)(q+2)}>1,\qquad
\frac{F(q,r+1)}{F(q,r)}
>2\frac{q+2}{q+3}\left(\frac q{q+1}\right)^2
\ge\frac{192}{175}>1.
\]
When $q=3$, direct calculation gives $F(3,r)>1$ for $1\le r\le4$.
For $r\ge4$, we have
\[
\frac{F(3,r+1)}{F(3,r)}
>\frac{9(r+4)}{8(r+5)}\ge1.
\]
Thus, $F(q,r)>1$, and the result follows.
\end{proof}

\begin{lemma}\label{lem:step-uniform}
For integers $t\ge4$ and $q\ge3$, we have $g_q(t)\ge g_3(t)>1$.
\end{lemma}
\begin{proof}
The ratio of consecutive terms, $g_{q+1}(t)/g_q(t)$, increases with $t$, so it suffices to check that it is greater than one when $t=4$.
The cases $q=3,4$ can be checked directly; when $q\ge5$, the inequality
$2^{q+1}-2>2(2^q-2)$ and the fact that $q^2(q+2)/(q+1)^3$ is increasing
imply that the ratio is greater than $28/27$. Thus, $g_q(t)$ increases with $q\ge3$.
Moreover, $g_3(t)\ge g_3(4)=2048/1875>1$, proving the result.
\end{proof}

We now give an initial bound for graphs on six vertices.
For $0\le i\le3$, let $J_i$ be the graph of order three with $i$ edges.
For each $i$, such a graph is unique up to isomorphism.

\begin{lemma}\label{lem:seeds}
For $0\le i,j\le3$, we have $Q(J_i\vee J_j)\ge\frac{34}{15}$.
\end{lemma}
\begin{proof}
Denote the Laplacian eigenvalues of $J_i$ by
$0=\lambda_1(J_i)\le\lambda_2(J_i)\le\lambda_3(J_i)$.
Since $\overline{J_i\vee J_j}=\overline{J_i}\sqcup\overline{J_j}$,
the formulas for the Laplacian spectra of complements and disjoint unions~\cite[Section~1.3.2 and Proposition~1.3.6]{BrouwerHaemers2012}
show that the eigenvalues of $J_i\vee J_j$ are $0,6$ and $3+\lambda_k(J_i),3+\lambda_k(J_j)$ for $k=2,3$. Write
$
p_i:=(3+\lambda_2(J_i))(3+\lambda_3(J_i))
$.
\[
\begin{array}{c|cccc}
i&0&1&2&3\\ \hline
(\lambda_2(J_i),\lambda_3(J_i))&(0,0)&(0,2)&(1,3)&(3,3)\\
p_i&9&15&24&36
\end{array}
\]
Applying the spectral form of the Matrix--Tree Theorem~\cite[Proposition~1.3.4]{BrouwerHaemers2012} then gives
$
\tau(J_i\vee J_j)
=p_ip_j.
$

Let $P(F,z)$ denote the chromatic polynomial of $F$.
Let $s(F,a)$ denote the number of partitions of $V(F)$ into $a$ nonempty stable sets,
and let $y_i$ be one when $i=0$ and zero otherwise.
Then $s(J_i,1)=y_i$, $s(J_i,2)=3-i$, and $s(J_i,3)=1$.
Every stable set in a join is contained in one of the factors, so
\[
P(J_i\vee J_j,z)=\sum_{a,b=1}^3s(J_i,a)s(J_j,b)(z)_{a+b}
\]
where $(z)_h=z(z-1)\cdots(z-h+1).$ Since $J_i\vee J_j$ has six vertices, Stanley's identity~\cite{Stanley1973} says that its number of acyclic orientations is $P(J_i\vee J_j,-1)$. Substituting the preceding expression for the chromatic polynomial gives
\[
\begin{aligned}
\alpha(J_i\vee J_j)
&=P(J_i\vee J_j,-1)\\
&=24[(i+2)(j+2)+5]+6y_i(j+1)+6y_j(i+1)+2y_iy_j\\
&\ge24[(i+2)(j+2)+5]\\
&\ge20(i+3)(j+3).
\end{aligned}
\]
Reference~\cite[Section~3.2]{CondeMerino2009} gives $\alpha^{*}(K_{3,3})=102$.
The graph $J_i\vee J_j$ is obtained from $K_{3,3}$ by adding $i+j$ internal edges.
Since $K_{3,3}$ is connected, a totally cyclic orientation is equivalent to a strong orientation~\cite[Section~19.3, p.~512]{BondyMurty2008}.
Take any totally cyclic orientation of $K_{3,3}$. If a new edge is oriented as $u\to v$,
strong connectivity ensures that the original orientation contains a directed path from $v$ to $u$, so the new edge also lies on a directed cycle;
the same applies to the reverse direction. The original edges remain on directed cycles, so the directions of the $i+j$ internal edges can be chosen independently,
giving $2^{i+j}$ totally cyclic extensions. Extensions of distinct original orientations remain distinct when restricted to $K_{3,3}$, so
$\alpha^{*}(J_i\vee J_j)\ge102\cdot2^{i+j}$.
On the other hand, direct calculation gives $p_i^2\le30(i+3)2^i$. Hence
\[
Q(J_i\vee J_j)=\frac{\alpha(J_i\vee J_j)\alpha^{*}(J_i\vee J_j)}{p_i^2p_j^2}\ge\frac{20(i+3)(j+3)\cdot102\,2^{i+j}}{p_i^2p_j^2}\ge\frac{20\cdot102}{30^2}=\frac{34}{15}.\qedhere
\]
\end{proof}

\begin{proof}[Proof of Theorem~\ref{thm:pair}]
Choose three vertices from each factor and retain all induced edges.
By Lemma~\ref{lem:seeds}, we have $Q\ge34/15$.
For $j=3,\ldots,a-1$, add vertices alternately to the two sides, increasing their orders from $(j,j)$ to $(j+1,j+1)$.
The product of the factors given by Lemma~\ref{lem:step} is
\[
g_j(j)g_{j+1}(j)
=\lambda_j\lambda_{j+1}
\left(\frac j{j+1}\right)^{4j-2}.
\]
Thus, when both sides have order $a$, we have $Q\ge D_a$.
Next, fix the smaller side and increase the order of the other side from $a$ to $b$.
The product of the additional factors telescopes to give
\[
\prod_{t=a}^{b-1}g_a(t)
=\lambda_a^{b-a}\left(\frac ab\right)^{2a-2}.
\]
At each step, the neighbours of the new vertex are exactly its neighbours already present in the target induced subgraph.
This proves~\eqref{eq:pair-bound}. Using Lemma~\ref{lem:step-uniform}, a straightforward calculation shows that the right-hand side of~\eqref{eq:pair-bound}
attains its minimum at $a=b=4$, with value $D_4=780759/524288$.
\end{proof}

\subsection{Proof of Theorem~\ref{thm:internal}}

\begin{proof}[Proof of Theorem~\ref{thm:internal}]
Start with $K_{a,b}$ on the given vertex sets and add the internal edges of $A$ and $B$ one at a time.
Every intermediate graph is connected and has no cut edges. Before an edge $uv$ is added within $A$,
every vertex of $B$ provides a path of length two from $u$ to $v$.
Let $H_0$ be the network obtained by retaining only these $b$ paths. Each path has resistance two,
so their effective resistance in parallel is $r_{uv}(H_0)=2/b$.
In passing from $H$ to $H_0$, the conductances of all other edges decrease from one to zero.
Effective resistance does not decrease when edge conductances decrease~\cite[Section~2.6, p.~44]{GhoshBoydSaberi2008}, so
$r_{uv}(H)\le r_{uv}(H_0)=2/b$.
By Lemma~\ref{lem:edge-growth}, we obtain
\[
Q(H+uv)\ge Q(H)\frac{2}{(1+2/b)^2}=Q(H)\eta_b.
\]
Similarly, adding an edge within $B$ gives a factor of $\eta_a$.
Multiplying these factors proves~\eqref{eq:internal-bound}.

When $a,b\ge5$, we have $\eta_a,\eta_b\ge\eta_5=50/49>1$,
which gives the uniform lower bound; whenever there is an internal edge, the inequality $Q(A\vee B)\ge Q(K_{a,b})$ is strict,
whereas equality holds when both factors are empty graphs.
\end{proof}

\subsection{Proof of Theorem~\ref{thm:universal}}

The following estimate controls the growth in the number of spanning trees at each step.

\begin{lemma}\label{lem:grounded}
Let $F$ have at least two vertices, and let $v$ be a vertex of minimum degree in $F$.
Let $J=F-v$, $H=K_2\vee J$, and $G=K_2\vee F$.
Write $n=v(H)$ and $d=d_F(v)+2$. Then
\begin{equation}\label{eq:grounded-growth}
\frac{\tau(G)}{\tau(H)}
\le\frac{n+1}{n}\,d
\left(1+\frac{3(d-1)}{2d^2}\right)^{d-2}.
\end{equation}
\end{lemma}
\begin{proof}
Let $\mathbf{M}=2\mathbf I_{v(J)}+\mathbf C(J)$ and $\delta=d-2$.
Delete from $\mathbf C(H)$ the row and column corresponding to one of the vertices of $K_2$.
By the Matrix--Tree Theorem and the identity
$\mathbf{M}\one_{v(J)}=2\one_{v(J)}$, we obtain
\[
\tau(H)=\det\begin{pmatrix}n-1&-\one_{v(J)}^T\\-\one_{v(J)}&\mathbf{M}\end{pmatrix}
=(n-1-\one_{v(J)}^T\mathbf{M}^{-1}\one_{v(J)})\det \mathbf{M}
=\frac n2\det \mathbf{M}.
\]
The same calculation for $G=K_2\vee F$ gives
$
\tau(G)=\frac{n+1}{2}\det(2\mathbf I_{v(F)}+\mathbf C(F)).
$
If $\delta=0$, then $d=2$ and $v$ is an isolated vertex in $F$, so
$\det(2\mathbf I_{v(F)}+\mathbf C(F))=2\det\mathbf M$.
Thus, $\tau(G)/\tau(H)=2(n+1)/n$, and equality holds in~\eqref{eq:grounded-growth}.
Assume henceforth that $\delta>0$.

Let $\mathbf{W}$ be the $v(J)\times\delta$ matrix whose columns are the standard unit vectors with coordinate set $V(J)$, namely $\mathbf e_u$ for $u\in N_F(v)$.
Write $\mathbf{A}=\mathbf{W}^T\mathbf{M}^{-1}\mathbf{W}$ and
$\mathbf{B}=\mathbf I_\delta-\one_\delta\one_\delta^T/d$.
Ordering the vertices according to $v,V(J)$ gives
\[
2\mathbf I_{v(F)}+\mathbf C(F)=
\begin{pmatrix}
d&-\one_\delta^T\mathbf{W}^T\\
-\mathbf{W}\one_\delta&\mathbf{M}+\mathbf{W}\mathbf{W}^T
\end{pmatrix}.
\]
Apply elementary row operations to eliminate the entries below $d$ in the first column, and then expand along that column to obtain
\[
\det(2\mathbf I_{v(F)}+\mathbf C(F))=d\det\left(\mathbf{M}+\mathbf{W}\mathbf{W}^T
-\frac1d\mathbf{W}\one_\delta\one_\delta^T\mathbf{W}^T\right)=d\det(\mathbf{M}+\mathbf{W}\mathbf{B}\mathbf{W}^T).
\]
For the matrix $\mathbf{B}$, the eigenvalue on the space spanned by $\one_\delta$ is $2/d$,
and the eigenvalue on its orthogonal complement is $1$. Thus, it is positive definite and has a symmetric square root $\mathbf{B}^{1/2}$.
Using the identity $\det(\mathbf I+XY)=\det(\mathbf I+YX)$ and the definition of $\mathbf A$, we further obtain
\[
\frac{\det(\mathbf{M}+\mathbf{W}\mathbf{B}\mathbf{W}^T)}{\det\mathbf{M}}
=\det(\mathbf I_{v(J)}+\mathbf{M}^{-1}\mathbf{W}\mathbf{B}\mathbf{W}^T)
=\det(\mathbf I_\delta+\mathbf{B}^{1/2}\mathbf{A}\mathbf{B}^{1/2}).
\]
Substituting the preceding expressions into $\tau(G)/\tau(H)$ gives
\begin{equation}\label{eq:grounded-determinant}
\frac{\tau(G)}{\tau(H)}
=\frac{n+1}{n}\,d\,
\det(\mathbf I_\delta+\mathbf{B}^{1/2}\mathbf{A}\mathbf{B}^{1/2}).
\end{equation}

For any $j\in V(J)$, let $x=\mathbf{M}^{-1}\mathbf e_j$.
If a minimum coordinate $x_i$ is negative, then
$(\mathbf{M}x)_i=2x_i+\sum_{w\in N_J(i)}(x_i-x_w)<0$, 
contradicting $(\mathbf{M}x)_i=(\mathbf e_j)_i\ge0$. Thus, all entries of $\mathbf{M}^{-1}$ are nonnegative.

In $H=K_2\vee J$, set the potentials at both vertices of $K_2$ to zero. Since every vertex of $J$ is adjacent to both of these vertices,
the matrix relating currents to potentials in $H$ on the coordinates of $J$ is $\mathbf M=2\mathbf I_{v(J)}+\mathbf C(J)$:
if the potential vector on $J$ is $\varphi$, then the vector of injected currents is $\mathbf M\varphi$.
If a unit current is injected only at $u\in V(J)$, then $\mathbf M\varphi=\mathbf e_u$, and hence
the potential at $u$ is $\varphi_u=\mathbf e_u^T\mathbf M^{-1}\mathbf e_u=(\mathbf M^{-1})_{uu}$.
After deleting all edges of $J$ not incident with $u$, there are two edges of unit resistance between $u$ and the two vertices at zero potential, with total conductance $2$.
For each $w\in N_J(u)$, after the deletions, $w$ is adjacent only to $u$ within $J$;
the two edges of unit resistance from $w$ to the vertices at zero potential are in parallel and have equivalent resistance $1/2$.
Thus, the branch from $u$ through $w$ consists of an edge of resistance $1$, namely $uw$, in series with this equivalent resistance,
and has total resistance $3/2$ and conductance $2/3$.
These branches are in parallel with the two direct edges, so, under unit current injection, the potential at $u$ after the deletions is $1/(2+2d_J(u)/3)$.
For $u\in N_F(v)$, the choice of a vertex of minimum degree gives $d_J(u)\ge\delta-1$.
After identifying the two vertices at zero potential as a single current sink, the potential at $u$ under unit current injection is the effective resistance from $u$ to that sink.
By Rayleigh monotonicity~\cite[Section~2.6, p.~44]{GhoshBoydSaberi2008}, deleting edges of $J$ does not decrease this resistance, so
\[
(\mathbf{M}^{-1})_{uu}\le\frac1{2+2d_J(u)/3}\le\frac3{2d}.
\]
Since all entries of $\mathbf{A}$ are nonnegative,
\[
\operatorname{tr}(\mathbf{A}\mathbf{B})
=\operatorname{tr}\mathbf{A}-\frac1d\one_\delta^T\mathbf{A}\one_\delta
\le\frac{d-1}{d}\operatorname{tr}\mathbf{A}
\le\frac{3\delta(d-1)}{2d^2}.
\]
The matrix $\mathbf{M}$ is positive definite and the columns of $\mathbf{W}$ are linearly independent, so $\mathbf{A}$ is positive definite.
Therefore, the $\delta$ eigenvalues of $\mathbf I_\delta+\mathbf{B}^{1/2}\mathbf{A}\mathbf{B}^{1/2}$ are positive,
and their sum is $\delta+\operatorname{tr}(\mathbf{A}\mathbf{B})$.
Applying the arithmetic--geometric mean inequality in~\eqref{eq:grounded-determinant} gives~\eqref{eq:grounded-growth}.
\end{proof}

\begin{proof}[Proof of Theorem~\ref{thm:universal}]
If $n=3$, then $F=K_1$, $G=K_3$, and $Q(G)=4/3$. Assume henceforth that $n\ge4$.
Successively delete a vertex of minimum degree from $F$ until only two vertices remain, and then add the deleted vertices back in reverse order.
Denote the join of order $s$ obtained in this process by $G_s$ for $4\le s\le n$, where $G_n=G$.
After each vertex is added in reverse order, the restored factor is precisely the graph from which that vertex was deleted, so the new vertex has minimum degree in it;
Lemma~\ref{lem:grounded} therefore applies at every step.

Let the degree of the new vertex added in passing from $G_s$ to $G_{s+1}$ be $d$, and write
$b_d=1+3(d-1)/(2d^2)$, $h_d=\lambda_d/b_d^{2d-4}$.
Proposition~\ref{prop:resistance} and Lemma~\ref{lem:grounded} give
\[
\frac{Q(G_{s+1})}{Q(G_s)}
\ge\frac{d^2\lambda_d}{\bigl(\tau(G_{s+1})/\tau(G_s)\bigr)^2}
\ge\left(\frac{s}{s+1}\right)^2h_d.
\]
It remains to verify that, for every integer $d\ge2$, we have $h_d\ge \frac{3}{2}$.
The inequality $2^{d+1}-2>2(2^d-2)$ gives
$\lambda_{d+1}/\lambda_d>2d^2(d+2)/(d+1)^3$; meanwhile, $b_d$ decreases with $d\ge2$.
When $d\ge5$,
\[
\frac{h_d}{h_{d+1}}
=\frac{\lambda_d}{\lambda_{d+1}}
\frac{b_{d+1}^{2d-2}}{b_d^{2d-4}}
<\frac{(d+1)^3}{2d^2(d+2)}\,b_{d+1}^2
\le\frac{108}{175}\left(\frac{29}{24}\right)^2
<1.
\]
Here we have used the fact that $(d+1)^3/(d^2(d+2))$ is decreasing and the bound $b_{d+1}\le b_6=29/24$.
Thus, $h_d$ is increasing for $d\ge5$. Direct calculation gives $h_2=h_3=3/2$ and $h_4,h_5>3/2$,
which proves the desired inequality.

Therefore, every step satisfies
\begin{equation}\label{eq:universal-step}
Q(G_{s+1})\ge Q(G_s)\frac32\left(\frac{s}{s+1}\right)^2.
\end{equation}
The initial graph $G_4$ of order four is either $K_2\vee\overline K_2$ or $K_4$.
The corresponding triples $(\alpha,\alpha^{*},\tau)$ are $(18,6,8)$ and $(24,24,16)$, respectively
\cite[proofs of Lemmas~3.3 and~3.4]{CondeMerino2009}, so $Q(G_4)\ge27/16$.
Combining this with~\eqref{eq:universal-step} gives
\[
Q(G)\ge\frac{27}{16}
\prod_{s=4}^{n-1}\left[\frac32\left(\frac{s}{s+1}\right)^2\right]
=\frac{27}{n^2}\left(\frac32\right)^{n-4}.
\]
\end{proof}

\subsection{Joins of at least three factors}

Taking complements turns the join operation into disjoint union. Thus, a graph can be expressed as a join of at least three factors with nonempty vertex sets
if and only if its complement has at least three components. We now combine Theorems~\ref{thm:pair}
and~\ref{thm:universal} to treat this case.

\begin{corollary}\label{thm:multi}
Let $G$ be a graph of order $n$ whose complement has at least three components. Then $Q(G)>1$.
When $n\ge16$, we have
\begin{equation}\label{eq:multi-bound}
Q(G)\ge\frac{27}{n^2}\left(\frac32\right)^{n-4}.
\end{equation}
\end{corollary}

To obtain Corollary~\ref{thm:multi}, we first give a symmetric lower bound for joins of two factors:

\begin{corollary}\label{cor:symmetric}
Let $a=v(A)\ge3$ and $b=v(B)\ge3$. Then
\[
Q(A\vee B)>
2\left(\frac83\right)^{a+b-6}
\left(\frac3a\right)^4\left(\frac3b\right)^4.
\]
\end{corollary}
\begin{proof}
Choose three vertices from each of $A,B$ to obtain an initial graph on six vertices as in Lemma~\ref{lem:seeds}.
When completing the two sides one vertex at a time, first increase each factor that needs to grow from order three to order four.
By Lemma~\ref{lem:step}, the lower-bound factor for the first step is $g_3(3)$; if both sides need to grow,
the lower-bound factor for the second step is $g_4(3)=(945/1024)g_3(3)$.
Thereafter, the factor being enlarged already has $t\ge4$ vertices at each step, so Lemma~\ref{lem:step-uniform}
gives a corresponding multiplicative factor of at least $g_3(t)$. Since $945/1024<1$, even if only one or neither of the two initial steps is needed,
this coefficient can still be used uniformly as a lower bound. Using $g_3(t)=(8/3)(t/(t+1))^4$, we obtain
\[
Q(A\vee B)\ge\frac{34}{15}\frac{945}{1024}
\prod_{t=3}^{a-1}g_3(t)\prod_{t=3}^{b-1}g_3(t)
>2\left(\frac83\right)^{a+b-6}
\left(\frac3a\right)^4\left(\frac3b\right)^4.
\]
\end{proof}

\begin{proof}[Proof of Corollary~\ref{thm:multi}]
Grouping the components of the complement into three groups, we may write $G=A\vee B\vee C$, where all three factors have nonempty vertex sets,
and assume without loss of generality that $v(A)\ge v(B)\ge v(C)$. If $v(A)\ge3$ and $v(B)+v(C)\ge3$,
we can combine $B,C$ into one factor and apply Theorem~\ref{thm:pair}.
If $v(A)\ge3$ and $v(B)+v(C)=2$, then $B\vee C=K_2$,
and Theorem~\ref{thm:universal} gives $Q(G)>1$. In the remaining case, $v(A)\le2$, and hence $n\le6$. If a factor contains an edge,
it is itself $K_2$; if two factors each consist of a single vertex, their join is also $K_2$.
Both cases are again covered by Theorem~\ref{thm:universal}.
In the remaining cases, all factors have no edges and at most one factor consists of a single vertex, leaving only
$K_{1,2,2}$ and $K_{2,2,2}$.
Starting with $K_{1,1,2}$, for which $(\alpha,\alpha^{*},\tau)=(18,6,8)$,
add a vertex of degree three and then a vertex of degree four. By Lemmas~\ref{lem:acyclic-orientations}
and~\ref{lem:cyclic-orientations}, the values of $\alpha\alpha^{*}$ for the two graphs are at least
$2592$ and $181440$, respectively. Their Laplacian spectra are $(0,3,3,5,5)$
and $(0,4,4,4,6,6)$, respectively, so their numbers of spanning trees are $45$ and $384$, respectively. Thus
$
Q(K_{1,2,2})\ge\frac{2592}{45^2}>1,
Q(K_{2,2,2})\ge\frac{181440}{384^2}>1.
$
This proves the strict inequality for every order.

Let $n\ge16$. By the classification above, $G$ can either be written as a join of two factors each of order at least three
or be written as $K_2\vee F$. In the former case, write $a,b$ for the orders of the two sides, so that $a+b=n$.
By Corollary~\ref{cor:symmetric} and $ab\le n^2/4$, we obtain
\[
Q(G)>2\left(\frac83\right)^{n-6}\left(\frac6n\right)^8.
\]
Consider the ratio of the right-hand side of this inequality to that of~\eqref{eq:multi-bound}.
At $n=16$, it is $134217728/129140163>1$. For $n\ge16$, this ratio grows from $n$ to $n+1$ by the factor
\[
\frac{16}{9}\left(\frac{n}{n+1}\right)^6>1\qquad(n\ge16).
\]
Thus,~\eqref{eq:multi-bound} holds in this case. If $G$ has the form $K_2\vee F$,
the inequality follows directly from~\eqref{eq:universal-bound}.
\end{proof}

\section*{Declaration of AI use}
The author used ChatGPT Astra(OpenAI) for language editing and stylistic suggestions during the preparation of the manuscript.
All mathematical statements, proofs, computations, references, and final wording were independently checked and approved by the author, who takes full responsibility for the content of the article.

\end{document}